\documentclass[11pt]{article}
\usepackage{geometry}

 \usepackage{graphics}
 \usepackage{graphicx}
 \usepackage{epsfig}
 \usepackage[all]{xy}

 \usepackage{amssymb,amsmath,amsthm}

 \usepackage{latexsym, epsfig, psfrag,eepic,colordvi,bm}
 \usepackage{graphics,color}
 \usepackage[all]{xy}
 \usepackage{epstopdf}
 \usepackage{appendix}

 \usepackage{graphicx}

\usepackage{graphicx,epsfig}
\usepackage{algorithmic, algorithm}

\usepackage{caption,subcaption}

\newtheorem{theorem}{Theorem}[section]

\newtheorem{proposition}[theorem]{Proposition}

\newtheorem{definition}{Definition}[section]

\DeclareMathOperator{\Deg}{deg}
\newcommand{\tp}{$[{\bf t}, {\bf p}]$}

\title{Network pearls climb stairs: a new game on graphs \\and its optimal solution}
\author{Ricky X. F. Chen~\footnote{ORCID: 0000-0003-1061-3049. Email: chenshu731@sina.com}\\
}

\date{}

\begin{document}
\maketitle

{\bf\noindent Abstract.}
Let $G$ be a graph, and $a$ and $b$ be integers.
Suppose there are infinite number of stairs, numbering level $0$, level $1$, etc.
How do you place every vertex of the graph $G$ on a level as high as possible
such that a vertex placed on level $i$ should have at least $i-b$ neighbors among the vertices
placed on level $i$ and above while have at least $i$ neighbors among the vertices placed on level $i-a$ and above?
This new game on graphs is referred to as ``network pearls climb stairs".
In this paper, we develop a more general theory, notably a correspondence between structure and dynamics, which particularly leads to the optimal solution of the game.
Moreover, as $a$ and $b$ vary, the corresponding level numbers for a vertex obviously provide a structural profile (or spectrum)
for the vertex the applications of which are worthy of further study.

{\noindent \bf Keywords:}
game on graphs, graph sequence, network dynamical system

{\noindent \bf MSC 2020:} 05C57, 90C27


{\section{Introduction}\label{sec:introduction}}

The main purpose of the paper is to develop a theory leading to the optimal
solution to a new game on graphs.
The game is described as follows.
Suppose $G$ is a simple graph, and $a$ and $b$ are non-negative integers.
Suppose there are infinite number of stairs, numbering level $0$, level $1$, etc.
How do you place every vertex of the graph $G$ on a level as high as possible
such that a vertex placed on level $i$ should have at least $i-b$ neighbors among the vertices
placed on level $i$ and above while have at least $i$ neighbors among the vertices placed on level $i-a$ (treated as zero if negative) and above?
This new game on graphs is referred to as ``network pearls climb stairs" (abbr.~NPCS).
See an example in Figure~\ref{npcs-exam}.
Although we will not discuss any applications, some possible applications of the NPCS game may be briefly pointed out here.
The stair indices for the vertices of a graph obviously encode some structural information of the graph.
Also, if two vertices always ``climb" to the same level when $a$ and $b$ vary, then in some sense they may
have very similar ``position" in the graph. These observations may be useful to some network problems.

\begin{figure}[H]
	\centering
	\begin{subfigure}{0.25\linewidth}
		\centering
		\includegraphics[width=1.0\linewidth]{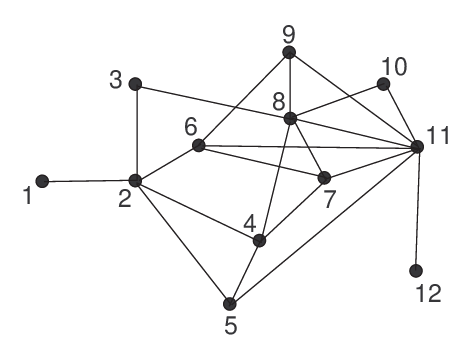}
		\caption{}
		\label{aa}
	\end{subfigure}\hskip -0pt
	\centering
	\begin{subfigure}{0.38\linewidth}
		\centering
		\includegraphics[width=1.0\linewidth]{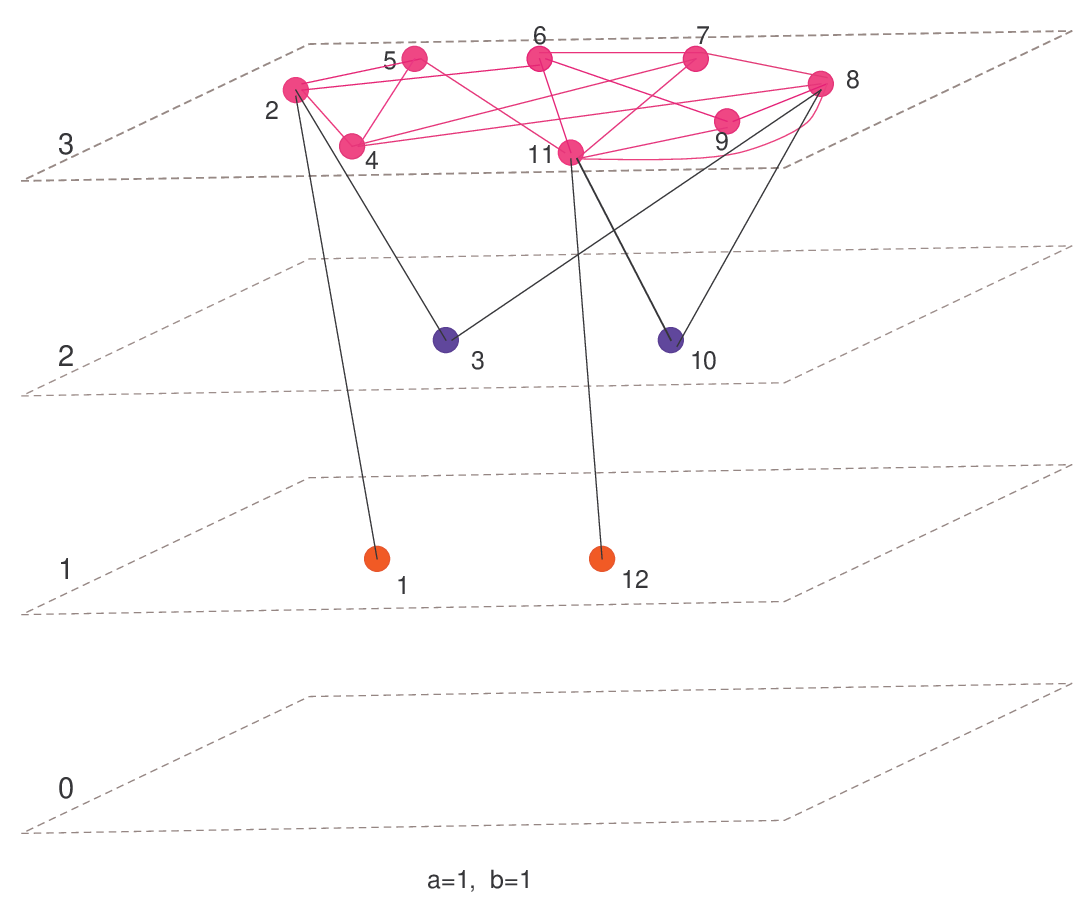}
		\caption{}
		\label{bb}
	\end{subfigure}	\hskip -14pt
	\centering
	\begin{subfigure}{0.38\linewidth}
		\centering
		\includegraphics[width=1.0\linewidth]{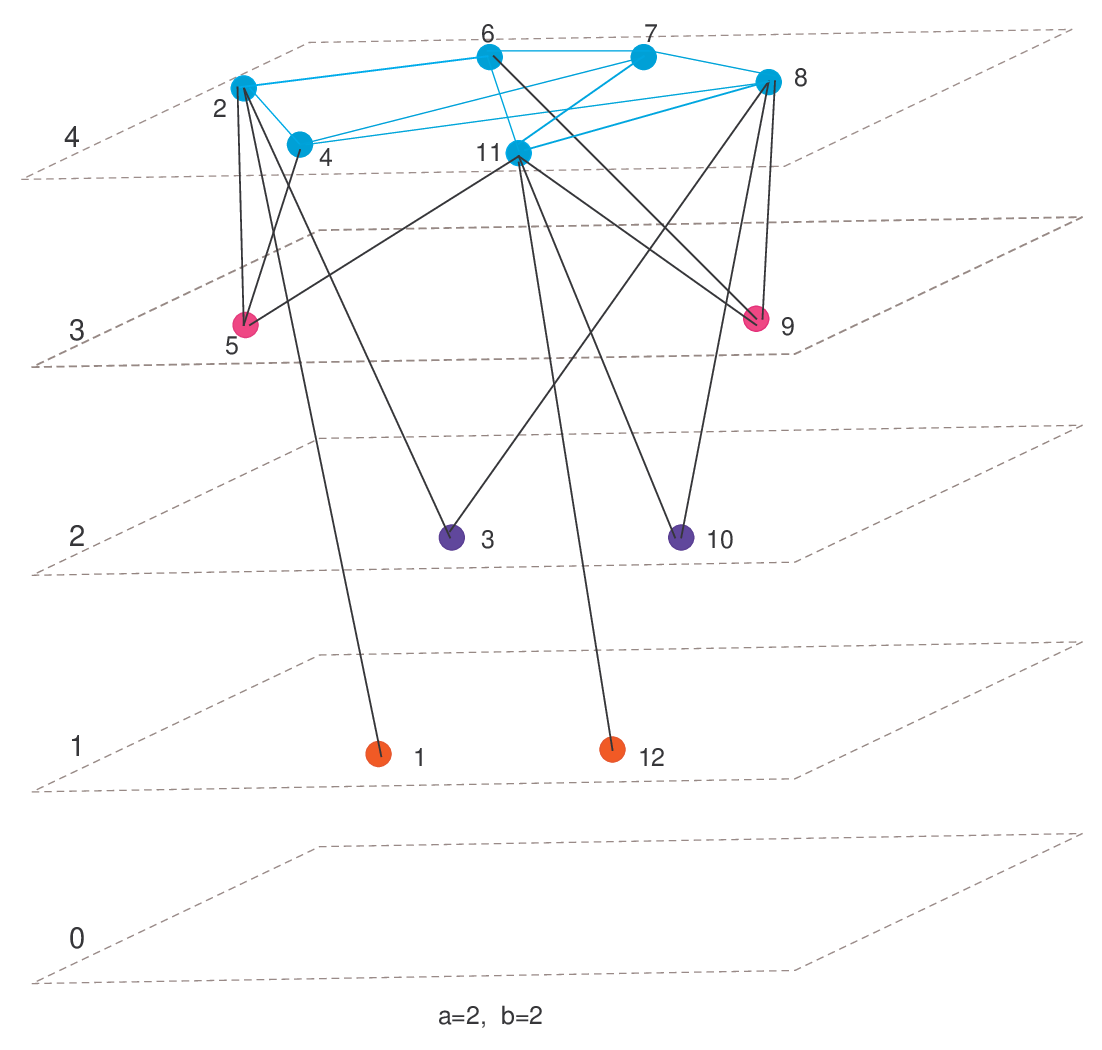}
		\caption{}
		\label{aaa}
	\end{subfigure}

	\caption{ (a) A graph $G$ of $12$ vertices. (b) The optimal solution of NPCS on $G$ with $a=1$ and $b=1$. (c) The optimal solution for
		$a=2$ and $b=2$.}
	\label{npcs-exam}
\end{figure}

The developed theory is a general new framework for analyzing graph structures, called \tp-filtered chain decomposition of graphs,
where $\bf{t}$ and $\bf{p}$ stand for certain function sequences.
Specifically, for a given graph $G$, each pair of such function sequences provides many subgraph sequences of $G$.
From such a subgraph sequence, every vertex of the given graph $G$ receives a rank. We next prove that certain extremal ranks for all vertices can be computed via searching fixed points (steady states) in certain dynamical systems on the graph.
As $\bf{t}$ and $\bf{p}$ vary, the corresponding extremal ranks for a vertex produces a high-resolution structural profile for the vertex.
Especially, the extremal ranks for some particular choice of $\bf{t}$ and $\bf{p}$ yields the desired optimal solution of the NPCS game.

The outline of the paper is as follows. In Section~$2$, we introduce the concept of \tp-filtered chain decomposition of graphs.
In Section~$3$, we present a novel means to compute the induced ranks of vertices from the chains.

\section{$[{\bf t},{\bf p}]$-filtered chain decomposition}

Let $G=(V,E)$ be a simple graph, where $V$ is the set of vertices (or nodes) and $E$ is its edge set.
We write $H\leq G$ if $H$ is a subgraph of $G$.
In the following, if not explicitly specified otherwise, a network of $n$ nodes has $[n]=\{1,2,\ldots, n\}$ as its node set. We sometimes write $v\in G$ for $v\in V(G)$.

\begin{definition}
	Let
	$
	{\bf t}=(t_1, t_2, \ldots, t_n),\, {\bf p}=(p_1, p_2, \ldots, p_n),
	$
	where $t_j$ and $p_j$ are integer-valued increasing functions defined over non-negative integers.
	Suppose $\mathcal{C}$ is a sequence of subgraphs of a network $G$ of $n$ nodes: $G_0, G_1 ,  G_2 ,  \ldots , G_m$. We call the sequence $\mathcal{C}$ a $[{\bf t},{\bf p}]$-filtered chain of $G$ if for any $0\leq i \leq m$, any vertex $v\in G_i$ has
	\begin{itemize}
		\item[(a)]  at least $p_v(i)$ neighbors in $G_i$, and 
		\item[(b)] at least $i$ neighbors in $G_j$ where $j=\max \{0, t_v(i)\}$.
	\end{itemize}
\end{definition}

 {\em Remark.} In this paper, given a graph $G$, when we say a vertex $v$ has $N$ neighbors in $G' \leq G$, we are referring to the adjacency relation in the original graph $G$ and $v$ is not necessarily contained in $G'$.
 No vertex is contained in an empty graph $\varnothing$ and any vertex has zero neighbors in $\varnothing$.

These chains aim at providing new insight into ``positions" of nodes in the global structure of the network $G$. Namely, we intend to characterize the internal mutual connection of a ``community" ($G_i$) in the global structure and the ability of the nodes there reaching the outside ($G_j$) of the community. 
An example of a \tp-filtered chain decomposition of a graph is presented in Figure~\ref{fig:tp-decomposition},
where $t_v(i)=i-1$ and $p_v(i)=i-1$ for any vertex $v$ of the graph.
	\begin{figure}[!htb]
	\centering
	\includegraphics[width=0.6\textwidth]{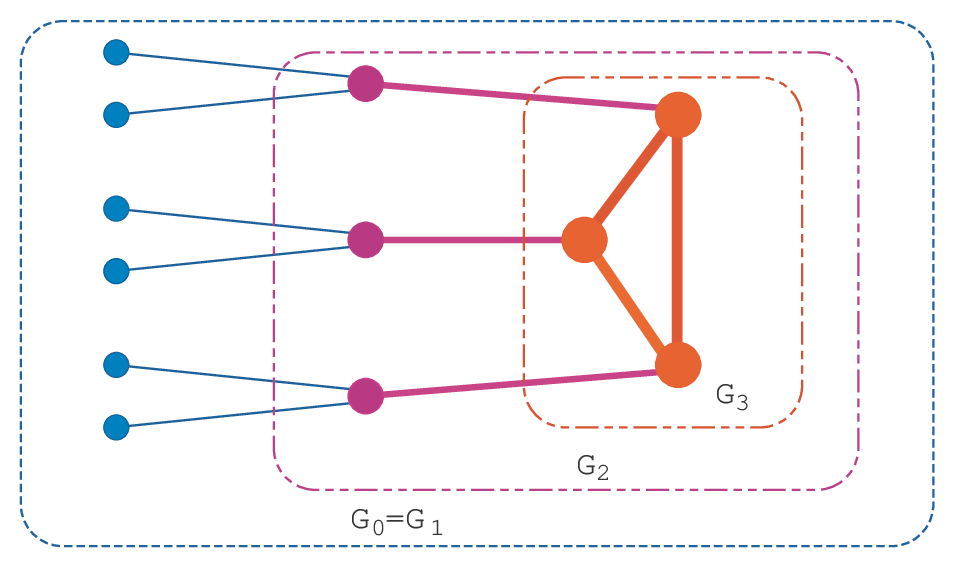}
	\caption{A \tp-filtered chain $G_0, G_1, G_2, G_3$ for a simple network $G=G_0$ of $12$ nodes and $12$ edges where $t_v(i)=i-1$ and $p_v(i)=i-1$ for any node $v$, i.e., a node contained in $G_i$ has at least $i-1$ neighbors in $G_i$ and at least $i$ neighbors contained in $G_j$ with $j=\max\{0, i-1\}$.}
	\label{fig:tp-decomposition}
\end{figure}

Note that given any \tp-filtered chain $G_0,  G_1, \ldots, G_m$ of $G$, the chain $G_0,  G_1, \ldots , G_m,  \varnothing , \ldots, \varnothing$
is also a \tp-filtered chain of $G$. Thus, given any \tp-filtered chain, there are many chains essentially the same as the given one.
In other words, some $\varnothing$'s are essentially necessary in order to rigorously satisfy the degree condition, some are not.
Due to this fact, it is necessary to distinguish the length and size of a chain.
For a \tp-filtered chain $\mathcal{C}: G_0,  G_1, \ldots, G_m$,
we call the number $m$ the length of the chain $\mathcal{C}$, and denote it by $L(\mathcal{C})=m$. The number $k$ such that $G_k\neq \varnothing$ while $G_{j}=\varnothing$ for all $j>k$ or $k+1>m$ is called the size of the chain $\mathcal{C}$, and denoted by $|\mathcal{C}|=k$.

For a node $v$ of $G$, we denote by $\Deg_G(v)$ (or $\Deg(v)$ if $G$ is clear) the degree of $v$ in $G$.
For some increasing function sequences ${\bf t}, {\bf p}$, there may be no \tp-filtered chains for $G$ at all.
The following proposition gives a sufficient and necessary condition for the existence of \tp-filtered chains.

\begin{proposition}[Existence]\label{prop:existance}
	Given $G=(V, E)$ and sequences of functions ${\bf t}, {\bf p}$,
	there exists a \tp-filtered chain of $G$ if and only if $p_v(0) \leq \Deg(v)$ for any $v\in V$.
\end{proposition}

\begin{proof}
	Consider the chain $G_0=G, \varnothing, \varnothing, \ldots $ where there are sufficient copies of the empty graph $\varnothing$. If $p_v(0)\leq \Deg(v)$, then $0 \leq \max\{0, 0+p_v(0)\} \leq \Deg(v)$. Then it is clear that for any $v\in G_0$, there are at least zero neighbors of $v$ in $G_j$ for any $j\geq 0$ (in particular for $j= \max \{0, t_v(0)\}$) and at least $p_v(0)$ neighbors of $v$ in $G_0$.
	Since $G_i=\varnothing$ for $i>0$, there is nothing to check and the degree condition is automatically satisfied. Thus the chain is a \tp-filtered chain.
	
	Conversely, suppose $G_0=G, G_1, \ldots, G_m$ is a \tp-filtered chain. Suppose vertex $v$ is contained in $G_i$ for some $i\geq 0$.
	Then, by definition, there must be at least $p_v(i)$ neighbors of $v$ in $G_i$. Since $G_i \leq G$, we have $p_v(i) \leq \Deg(v)$, which implies $p_v(0) \leq \Deg(v)$ due to $p_v$ being an increasing function. This completes the proof.
\end{proof}

Hereafter, we assume $p_v(0) \leq \Deg(v)$ for any node $v$ unless they are
explicitly specified.
When exist, there may be many \tp-filtered chains of a given graph.
We will not explore the whole space of these chains in this paper, but
focus on some special chains.

\begin{definition}[Maximal chains]
	A \tp-filtered chain $\mathcal{C}$ of $G$: $G_0, G_1, \ldots, G_m$ is called maximal if there does not exist a \tp-filtered chain $\mathcal{C}': G'_0, G'_1, \ldots, G'_{m'}$ satisfying either (i) $|\mathcal{C}'|> |\mathcal{C}|$, or (ii) $|\mathcal{C}'|= |\mathcal{C}|$ and for some $ 1\leq i \leq |\mathcal{C}|$, $G_i < G'_i$.
\end{definition}

Note that merging two \tp-filtered chains entry-wisely gives another \tp-filtered chain.
Therefore, it is not hard to show the following crutial uniqueness result.

\begin{proposition}[Uniqueness]
	For $G=(V, E)$, 
	 there exists a unique maximal \tp-filtered chain of $G$ whenever $p_v(0) \leq \Deg(v)$ for any $v\in V$.
\end{proposition}

Moreover, we have a nesting property for the maximal \tp-filtered chain.

\begin{proposition}[Nesting property]
	For $G=(V, E)$, suppose $G_0, G_1, \ldots, G_m$ is the 
	maximal \tp-filtered chain of $G$. Then, we have $G_0 \geq G_1 \geq \ldots \geq G_m$ and $G_0=G$.
\end{proposition}
\begin{proof}
	Suppose $G_0, G_1, \ldots, G_m$ is a \tp-filtered chain of $G$.
	For $i\geq 0$, let 
	$$
	G'_i=\bigcup_{j\geq i} G_j.
	$$
	The sequence of $G'_i$ obviously has the nesting property.
	It then suffices to show that $G'_0, G'_1,
	\ldots, G'_m$ and $G, G'_1,
	\ldots, G'_m$ are also \tp-filtered chains of $G$.
	Suppose $v\in G'_i$. We distinguish two cases to prove the former first.
	If $v \in G_i$, by construction, $v$ has at least $i$ neighbors in $G_j \leq G'_j$ thus in $G'_j$
	while has at least $p_v(i)$ neighbors in $G_i \leq G'_i$ thus in $G'_i$.
	If $v \notin G_i$, then $v \in G_k$ for some $k>i$; Otherwise, it is impossible.
	By construction, $v$ has at least $k$ neighbors in $G_j$ for $j=\max\{0, t_v(k)\} \geq \max\{0, t_v(i)\}=j'$.
	Since $G'_{j'} \geq G'_j \geq G_j$ and $i<k$, $v$ has at least $i$ neighbors in $G'_{j'}$. On the other hand,
	$v$ has at least $p_v(k)$ neighbors in $G_k$ by construction. Since $G'_i \geq G'_k \geq G_k$ and $p_v(i)\leq p_v(k)$,
	we conclude that $v$ has at least $p_v(i)$ neighbors in $G'_i$. 
	In summary, $G'_0, G'_1,
	\ldots, G'_m$ is a \tp-filtered chain of $G$. Next, replacing $G'_0$ with the graph $G$ itself obviously yields a \tp-filtered chain, completing the proof.
\end{proof}

Note that the nesting property is not necessary for a chain to be a \tp-filtered chain.
We also remark that the maximal D-chain (dendrite chain) of order $t\leq 0$ of a network $G$ introduced by Chen, Bura and Reidys~\cite{d-spec}
is equivalent to the maximal \tp-filtered chain of $G$ with $t_v(i)=i+t$ and $p_v(i)=0$ for all $v\in G$.

In the maximal \tp-filtered chain, if a vertex $v\in G_k$ and $v\notin G_{k+1}$, then we denote $C_{\bf t, p}(v)=k$.
That is, the maximal \tp-filtered chain induces a unique rank for every vertex. 
Apparently, $G_k$ is just the vertex-induced subgraph of $G$ by the set of vertices $v$ with $C_{{\bf t,p}}(v)=k$.

Finally, we relate to the NPCS game introduced at the beginning of the paper.
For a given graph $G$ and $a, b$ being non-negative integers,
let $G_i$ denote the subgraph induced by the vertices placed on level $i$
and above in a solution of the game.
We claim that the NPCS game has a unique optimal solution, and the sequence
$G_0, G_1, \ldots$ from the optimal solution is exactly the maximal \tp-filtered chain
of $G$ for $t_v(i)=i-a$ and $p_v(i)=i-b$ for any $v\in G$.
The proof of the claim is left to the interested reader.
Consequently, to find the optimal solution of the NPCS game is to compute
 $C_{\bf t, p}(v)$ for all $v \in G$ which
 will be discussed in the forthcoming section.

\section{Computation}

Inspecting the definition of \tp-filtered chains, it is not clear how to obtain such chains.
In fact, we do not know any efficient approach for obtaining all chains at present.
However, we do find an approach towards obtaining the maximal \tp-filtered chain if exists.
The approach is based on a novel connection discovered between the maximal \tp-filtered chain
of a graph $G$ and a certain fixed point of some network dynamical system on $G$.

A network dynamical system~\cite{linear2,linear1,kauf,vonn,wolf,rei5} is concerned with
the dynamics generated when nodes in the network (i.e., graph) update their states following a system update schedule and their respective rules.
Von Neumann's cellular automata~\cite{vonn} are such dynamical systems.

Let $P$ denote a finite set of states a node in a graph $G$ of $n$ vertices may have and let $x_i$ denote the state of node $i$.
Then, a system state is a vector $(x_1, x_2, \ldots, x_n) \in P^n$.
For any vertex $i$, suppose there is
a (local) function $f_{i}$ with the states of the neighbors of $i$
(and itself) in $G$ as input arguments.
An update schedule is
an infinite sequence of subsets of vertices $W=W_1W_2\cdots$.
Suppose the initial system state at time $t=0$ is ${\bf x}^{(0)}=(x^{(0)}_1, x^{(0)}_2, \ldots, x^{(0)}_n)$.
For $j>0$, the system state ${\bf x}^{(j)}$
at time $t=j$ follows from that the nodes contained in $W_j$ update their states via
applying their respective associated functions to the states of their respective neighbors (and themselves) in ${\bf x}^{(j-1)}$
while the states of the nodes not contained in $W_j$ stay unchanged.
We denote this dynamical system by $[G,f,W]$, and we denote by
$[G,f,W]^{(j)}({\bf x})$ the system state at time $t=j$ when the system starts
from the initial state ${\bf x}$ at time $t=0$.

({\bf Fixed point}) A system state ${\bf z} \in P^n$ is called a fixed point (or steady state) if the state of vertex $i$ remains changed when
updating the state of vertex $i$ from ${\bf z}$ by applying the function $f_i$ for any $i$.
Note that our definition of fixed points does not depend on the choice of update schedules.
Obviously, for any update schedule $W$, if ${\bf z}$ is a fixed point, then there exists a system state ${\bf x}$
and $k> 0$ such that for any $j>k$,
$$
[G,f,W]^{(j)}({\bf x})=[G,f,W]^{(k)}({\bf x})={\bf z}.
$$
In this case, we say ${\bf x}$ is
reaching the fixed point ${\bf z}$ in the dynamical system $[G,f,W]$.

Suppose there is a linear order `$\leq $' on $P$. Let $P^q=\{(x_1,x_2,\dots, x_q): x_j \in P,\, 1\leq j \leq q\}$. Then, 
there is a natural partial order on $P^q$ as follows:  $(x_1,x_2,\dots , x_q)\leq (y_1,y_2,\dots, y_q)$ if for all $1\leq j \leq q$, $x_j\leq y_j$ in $P$.
A function $g: P^q\rightarrow P$ is called monotonically increasing, if for any ${\bf x\leq y}$ in $P^q$, we have $g(x)\leq g(y)$ in $P$.
Fox example, the binary functions `AND' and `OR' on $P^q=\{0,1\}^q$ are monotonically increasing.

The local function $f_{i}: (x_i, x_{k_1},x_{k_2},\dots, x_{k_i})\mapsto x'_i$ is called contractive (where
$k_1,\ldots, k_i$ are the neighbors of $i$), if
for any argument $(x_i, x_{k_1},x_{k_2},\dots, x_{k_i})\in P^{k_i+1}$, $x'_i \leq  x_i$ holds.
It is easy to check that if $f_{i}$ is the binary function `AND', then it is contractive under the assumption $0<1$.
A dynamical system where the local function associated with every vertex is monotonically increasing and contractive is called a monotone-contractive (M-C) system.

An update schedule $W=W_1 W_2 \cdots$ is called fair if for any $k>0$ and any $v\in G$,
there exists $l>k$ such that $v\in W_l$. 
For example, if $W_i=[n]$ for any $i$, then $W$ is fair; if each $W_j$ contains a single vertex and $W_{i+1}W_{i+2}\cdots W_{i+n}$
is a permutation on the vertex set $[n]$ for any $i$, then $W$ is fair.
The following results are relevant.

\begin{proposition}[\cite{d-spec}]\label{main-thm-2}
	Suppose the system state ${\bf x} \in P^n$ is reaching the fixed point ${\bf z} \leq {\bf x}$ in a dynamical system with monotonically increasing local functions. Then,
	any state ${\bf y} \in P^n$ such that ${\bf z} \leq {\bf y} \leq {\bf x}$  or ${\bf y}< {\bf x}$ but not comparable to ${\bf z}$ is not a fixed point of the dynamical system.
\end{proposition}

\begin{proposition}[\cite{d-spec}]\label{main-thm-1}
	For any two fair update schedules $W$ and $W'$, a system state ${\bf x} \in P^n$ is reaching the same fixed point ${\bf z}\leq {\bf x}$ under the two M-C systems
	$[G,f,W]$ and $[G,f,W']$.
	Furthermore, any system state ${\bf y} \in P^n$ such that ${\bf z}\leq {\bf y} \leq {\bf x}$ is reaching the fixed point ${\bf z}$.
\end{proposition}

As a consequence of Proposition~\ref{main-thm-1}, the exact form of the update schedule does not matter when come to discussing fixed points in M-C systems as long as it is fair. Therefore,
we shall not explicitly specify the update schedules of M-C systems unless it is necessary.

({\bf $[{\bf t},{\bf p}]$-system})  Let $G$ be a graph on vertices in $[n]$.
Given a pair of integer-valued increasing function sequences ${\bf t}=(t_1,\ldots, t_n)$ and ${\bf p}=(p_1,\ldots, p_n)$, we define a particular dynamical system on $G$, called the $[{\bf t},{\bf p}]$-system on $G$, as follows:
each vertex of $G$ can have a state from the set $[n]$, and the function $f_v$ at a vertex $v$ returns the maximum $k \geq 0$ such that there are $k$ of the neighbors of $v$ with (state) values at least $t_v(k)$ while $p_v(k)$ of them with values at least $k$, and a specified update schedule $W$.

Here is an example of how the local functions work. Suppose for $i\geq 0$, $t_v(i)=i-2$ and $p_v(i)=i-1$, and $\{2,4,4,5,3\}$
is the (multi)set of values of the neighbors of $v$. Then,
$f_v$ returns $4$, i.e., $v$ has $4$ neighbors with values at least $4-2=2$ and has $4-1=3$ neighbors
with values at least $4$, and the number $4$ cannot be further increased to $5$ or even larger.

It turns out that these systems are useful in computing maximal chains,
and $W$ can be an arbitrary fair update schedule thus sometimes not specified.

\begin{proposition}
Suppose $G$ is a graph on $[n]$ and ${\bf d}=\big(\Deg(1),\Deg(2),\dots, \Deg(n)\big)$.
Given function sequences ${\bf t}$ and ${\bf p}$,
let 
$$
\mathcal{Q}=\bigcup_{W} \{z: z=[G,f,W]^{(i)}({\bf d}) \mbox{ for some $i\geq 0$}\},
$$
where the union is taken over $[{\bf t},{\bf p}]$-systems $[G,f,W]$ for all possible fair update schedules $W$.
Then, for any $v\in G$, its associated function $f_v$ is contractive with respect to any system state
contained in $\mathcal{Q}$.
Moreover, for all fair update schedules $W$, all states in $\mathcal{Q}$ will reach the same fixed point in the \tp-systems $[G,f,W]$.
\end{proposition}

\begin{proof}
	For any fixed update schedule $W=W_1W_2\cdots$ and $v \in G$, we will show that $f_v$
	is contractive with respect to $[G,f,W]^{(i)}({\bf d})$ for all $i\geq 0$ by induction.
	Note that $[G,f,W]^{(0)}({\bf d})={\bf d}$.
	By definition, when updating $v$ from ${\bf d}$, the returned state is obviously no larger than $\Deg(v)$.
	Thus, $f_v$ is contractive w.r.t.~${\bf d}$.
	
	Suppose $f_u$ is contractive w.r.t.~$[G,f,W]^{(i)}({\bf d})$ for any $u\in G$.
	Consequently, there holds
	$$
	[G,f,W]^{(i+1)}({\bf d}) \leq [G,f,W]^{(i)}({\bf d}),
	$$
	no matter which vertices are contained in $W_{i+1}$.
	Note that although the local function $f_v$ only uses the states of the neighbors of $v$ as the input,
	it can be easily viewed as (or extended to) a function taking the states of all vertices (i.e., a system state)
	as the input, and we will still denote the ``extended" function by $f_v$ if no confusions arise.
	It is also not hard to realize that $f_v$ is monotonically increasing.
	Thus,
	$$
	f_v([G,f,W]^{(i+1)}({\bf d}) )\leq f_v([G,f,W]^{(i)}({\bf d})) \leq [G,f,W]^{(i)}({\bf d})[v],
	$$
	where ${\bf x}[v]$ indicates the $v$-th coordinate in ${\bf x}$.
	We next distinguish two cases.
	If $v\notin W_{i+1}$, then we clearly have 
	$$
		[G,f,W]^{(i+1)}({\bf d})[v] = [G,f,W]^{(i)}({\bf d})[v].
	$$
	The last two relations then imply that $f_v$ is contractive w.r.t.~$	[G,f,W]^{(i+1)}({\bf d})$.
	If $v \in W_{i+1}$, we have 
	$$	
	[G,f,W]^{(i+1)}({\bf d})[v]=f_v([G,f,W]^{(i)}({\bf d})).
	$$
	Together with the third last relation, they also imply $f_v$ is contractive w.r.t.~$	[G,f,W]^{(i+1)}({\bf d})$.
	Hence, with respect to the portion of system states space $\mathcal{Q}$,
	the \tp-system is an M-C system.
	In view of Proposition~\ref{main-thm-1} and~\ref{main-thm-2}, the last statement follows,
	completing the proof.
	\end{proof}

\begin{theorem}[Structure profile/Dynamics Correspondence]\label{thm:main3}

	Suppose $G$ is a graph on $[n]$. Let ${\bf d}=\big(\Deg(1),\Deg(2),\dots, \Deg(n)\big)$.
	Then, in the
	\tp-system on $G$ with an arbitrary fair update schedule, the state ${\bf d}$ is reaching a stable state $C^{\bf t,p}=(C^{\bf t,p}_1,C^{\bf t,p}_2,\ldots, C^{\bf t,p}_n)$, and for any $i \in G$, 
	$$
	C^{\bf t,p}_i= C_{\bf t, p}(i).
	$$

\end{theorem}

\begin{proof}
	Let ${\bf z}=\big( C_{{\bf t},{\bf p}}(1)\ldots, C_{{\bf t},{\bf p}}(n)\big)$.
	Then, in view of Proposition~\ref{main-thm-1} and~\ref{main-thm-2}, it suffices to prove that
	the state ${\bf z}$ is a fixed point of the $[{\bf t},{\bf p}]$-system
	and the state ${\bf d}$ is reaching ${\bf z}$.
	We shall first prove:
	
	{ \emph{Claim}~$1$.} The state ${\bf z}$ is a fixed point of the $[{\bf t},{\bf p}]$-system.

	Let $\mathcal{C}: G_0\geq G_1\geq \cdots $ be the maximal $[{\bf t}, {\bf p}]$-separate chain of $G$.
	First, suppose $C_{{\bf t},{\bf p}}(v)=i$ for a vertex $v$. By definition, $v$ belongs to $G_i$ but not $G_{i+1}$,
	which tells us:
	\begin{itemize}
		\item[ (i)] there exist $i$ neighbors of $v$ contained in $G_{j}$ ($j=\max\{0, t_v(i)\}$) and $p_v(i)$ neighbors of $v$ contained in $G_i$. Note that for any $u$ among these said neighbors in $G_j$, we have $C_{{\bf t},{\bf p}}(u)\geq t_v(i)$, and for any $u'$ among these neighbors contained in $G_i$ we have $C_{{\bf t},{\bf p}}(u')\geq i$ by definition. Thus, among the neighbors of $v$, there exist at least $i$ of them with values at least $t_v(i)$ and at least $p_v(i)$ of them with values at least $i$ in ${\bf z}$. This leads to that $f_v({\bf z})\geq i=C_{{\bf t},{\bf p}}(v)$;
		\item[(ii)] it is impossible that $i+1$ neighbors of $v$ are contained in $G_{j'}$ ($j'=\max\{0, t_v(i+1)\}$) and $p_v(i+1)$ neighbors of $v$ are contained in $G_{i+1}$. Otherwise, the chain $G_0\geq \cdots \geq G_i\geq G_{i+1}\bigcup \{v\}\geq G_{i+2}\geq \cdots$ yields a $[{\bf t}, {\bf p}]$-filtered  chain, which contradicts the maximality of $\mathcal{C}$. Hence, the case that $i+1$ of the neighbors of $v$ have values at least $t_v(i+1)$ and $p_v(i+1)$ of the neighbors with values at least $i+1$ in ${\bf z}$ cannot happen. This implies
		$f_v({\bf z})< i+1$.
	\end{itemize}
	From (i) and (ii), we conclude $f_v({\bf z})=i=C_{{\bf t},{\bf p}}(v)$ for any $v$. Therefore, ${\bf z}$ is a fixed point.
	
	We next shall show that ${\bf d}$ is reaching ${\bf z}$. If ${\bf z=d}$, we are done.
	Otherwise we clearly have ${\bf z<d}$.
	For this case, in the light of
	Proposition~\ref{main-thm-2}, it suffices to show that any state ${\bf y\leq d}$ such that ${\bf y>z}$ or ${\bf y}$ is uncomparable with ${\bf z}$ is not a fixed point.
	
	We prove by contradiction.
	Suppose ${\bf y}$ is such a state which is a fixed point. Then there must be a coordinate indexed by some $v$ that satisfies $y_v> C_{{\bf t},{\bf p}}(v)$. Consider the sequence of subgraphs induced by the sequence of sets of vertices $S_0 \supseteq S_1\supseteq \cdots \supseteq S_{m}$ for a sufficient large number $m$ which are iteratively constructed as follows:
	\begin{itemize}
		\item[(1)] set $S_{y_v}=\{v\}$ and $S_r=\varnothing$ for $r\neq y_v$;
		\item[(2)] for $r$ from $0$ to $m$, if $u\in S_r$, then set
		\begin{align*}
			 S_{\max\{0,t_u(r)\}}&=S_{\max\{0,t_u(r)\}}\bigcup 
			\{w: \mbox{$w$ is a neighbor of $u$ and $y_w\geq t_u(r)$}\},\\
			 S_{r}&=S_{r}\bigcup \{w: \mbox{$w$ is a neighbor of $u$ and $y_w\geq r$}\},
		\end{align*}
		and for $0\leq r \leq m$, set $S_r=\bigcup_{j=r}^m S_j$;
		\item[(3)] iterate (2) until the sequence  $S_0 \supseteq S_1\supseteq \cdots \supseteq S_{m}$ becomes stable, i.e., $S_r$ does not change for any $0\leq r \leq m$ when further executing (2).
	\end{itemize}
	Clearly, by construction we have $S_{r}\subseteq S_{r-1}$.
	By abuse of notation, we denote by $S_r$ the subgraph induced by $S_r$ as well. Then we have a chain of graphs $S_0\geq S_1\geq \cdots \geq S_{m}$.
	Let $\tilde{{\bf t}}$ be the restriction of ${\bf t}$ to that of the vertices in $S_0$ and $\tilde{{\bf p}}$ be the restriction of ${\bf p}$ to that of the vertices in $S_0$.
	We proceed to show that
	
	{\emph{Claim}~$2$.} The chain $S_0\geq S_1\geq \cdots \geq S_{m}$ gives a $[\tilde{{\bf t}},\tilde{{\bf p}}]$-filtered chain of the graph $S_0\leq G$. That is, for any $u\in S_0$, if $u\in S_r$, then there are $r$ neighbors of $u$ contained in $S_{j}$ where $j=\max\{0, t_u(r)\}$ and $p_u(r)$ neighbors of $u$ contained in $S_r$.

	Since ${\bf y}$ is a fixed point by assumption, there exist $y_u$ neighbors of $u$
	whose corresponding values in ${\bf y}$ are at least $t_u(y_u)$ and $p_u(y_u)$ neighbors whose corresponding values in ${\bf y}$ are at least $y_u$, but there are not $y_u+1$ neighbors of $u$
	whose corresponding values in ${\bf y}$ are at least $t_u(y_u+1)$ and $p_u(y_u+1)$ neighbors whose corresponding values in ${\bf y}$ are at least $y_u+1$.
	By construction of (2), if $u\in S_r$, then $r\leq y_u$.
	Furthermore, since $u$ has at least $y_u$ neighbors $w$ such that $y_w \geq t_u(y_u) \geq t_u(r)$,
	these vertices $w$ are contained in $S_{\max\{0, t_u(r)\}}$. Analogously, there are at least $p_u(y_u)\geq p_u(r)$
	neighbors of $u$ contained in $S_r$.
	Accordingly, the chain $S_0\geq S_1\geq \cdots \geq S_{m}$ gives a $[\tilde{{\bf t}},\tilde{{\bf p}}]$-separate chain, whence Claim~$2$.

	In view of Claim~$2$, we have $C_{\bf t,p}(v)\geq y_v$ since $v\in S_{y_v}$, which yields a contradiction.
	Hence, ${\bf y}$ cannot be a fixed point.
	According to Proposition~\ref{main-thm-2}, ${\bf d}$ cannot reach a fixed point that is smaller than ${\bf z}$ either.
	Therefore, ${\bf d}$ is reaching ${\bf z}$, and the proof follows.
\end{proof}

According to Theorem~\ref{thm:main3}, the maximal \tp-separate chain if needed can be immediately constructed once the fixed point of the \tp-system is obtained.
It is also worthy of pointing out that the computation can be implemented distributively since the vertices are not required to synchronously update their states
when searching the fixed point.

For $	{\bf t}=(t_1, t_2, \ldots, t_n)$ and $	{\bf t'}=(t'_1, t'_2, \ldots, t'_n)$,
we write ${\bf t} \leq {\bf t'}$ if $t_v(i) \leq t'_v(i)$ for all $v$ and $i$.
Then, the following proposition holds and its proof is left to the interested reader.

\begin{proposition}
	Let $G$ be a graph of $n$ vertices. Suppose ${\bf t} \leq {\bf t'}$ and ${\bf p} \leq {\bf p'}$.
	Then, we have
	$
	C^{\bf t',p'} \leq  C^{\bf t,p}.
	$
	Moreover, $C^{\bf t,p}$ is reaching $C^{\bf t',p'}$ in the $[{\bf t',p'}]$-system of $G$.
\end{proposition}

Finally, \tp-filtered chains of a graph can be generalized. For instance, we may
consider \tp-filtered chains of distance $l$, where we replace the requirements of neighbors
with that of distance $l$ neighbors. The details are omitted here.


\section*{Declarations}

\noindent{\bf Conflict of Interest:} None.


\end{document}